\documentclass[12pt,a4paper]{amsart}

\usepackage{graphicx} 

\usepackage{color}
\usepackage{amsmath,amsthm,amsfonts,amssymb}
\usepackage{centernot}

\newtheorem{theorem}{Theorem}

\newtheorem{corollary}{Corollary}
\theoremstyle{definition}

\newtheorem{definition}{Definition}
\newtheorem*{Definition}{Definition}
\newtheorem*{remark}{Remark}
\newtheorem{proposition}{Proposition}

\title{$k$-type entropy of $\mathbb{Z}^d$ actions}
\author{Anshid Aboobacker, Sharan Gopal}

\begin{document}

\begin{center}

\large{\textbf{$k$-type entropy of $\mathbb{Z}^d$-actions}}

\small{Anshid Aboobacker, Sharan Gopal}

\textit{Department of Mathematics, BITS Pilani Hyderabad Campus}

anshidaboobackerk@gmail.com, sharan@hyderabad.bits-pilani.ac.in
\end{center}

\section*{Abstract}
We introduce the concept of \(k\)-type entropy for dynamical systems generated by \(\mathbb{Z}^d\)-actions on compact metric spaces. We investigate its fundamental properties and establish connections with classical entropy and other \(k\)-type dynamical notions. The $k$-type entropy of some $\mathbb{Z}^2$-actions on a two dimensional torus is also calculated.

\noindent\textbf{Keywords:} entropy, $k$-type entropy, group actions

\noindent\textbf{2020 MSC:} 37B05, 37C85


\section{Introduction}

A dynamical system $(X,T)$ is a compact metric space $X$ together with a group action $T: G\times X \rightarrow X$.  In our paper, we study about the dynamical systems with $G=\mathbb{Z}^d$ with $d$ being a positive integer.

To study eventual behaviours of orbits in a dynamical system, when $d=1$, we generally take the direction as $n$ tends to the positive infinity. In the case of $\mathbb{Z}^d$-actions, Oprocha in his 2007 paper \cite{oprocha2007chain}, gave the notion of $k$-type order on $\mathbb{Z}^d$, where $k\in \{1,2,\dots,2^d\}$. He introduced the notions of $k$-type limit sets, $k$-type limit prolongation sets and $k$-type transitivity for $\mathbb{Z}^d$-actions. Shah and Das \cite{shah2015note,shah2015different} further developed on this to define $k$-type periodic points, $k$-type sensitivity, $k$-type Devaney chaos, $k$-type Li Yorke pairs, etc. They also studied about preservation of systems under conjugacies and about induced systems on hyperspaces.

Developing this further, we in our earlier paper \cite{aboobacker2025ktype} have defined and studied $k$-type proximal pairs, $k$-type asymptotic pairs, $k$-type Li Yorke Sensitivity, $k$-type Li Yorke pairs and $k$-type Li-Yorke chaos and various relations between them. We also showed that all these notions are preserved under conjugacies and looked into how these notions work in the induced $\mathbb{Z}^d$-actions. In this paper, we define $k$-type entropy for $\mathbb{Z}^d$-actions and study its various properties.

In the next section, we give the definitions and related results of topological entropy of $\mathbb{Z}$-actions as given in \cite{brin2002introduction}. We also mention some basic facts and also fix notations for $k$-type notions of $\mathbb{Z}^d$-actions in this section. Section 3 contains the definitions of $k$-type entropy and some results. In the final section, we calculate the entropy of certain $\mathbb{Z}^2$-actions defined by two commuting hyperbolic matrices on a two dimensional torus.

\section{Preliminaries}
As mentioned above, we will first give the definitions and some results for entropy of $\mathbb{Z}$-actions; we follow \cite{brin2002introduction} for all these. In this section, $X$ denotes a compact metric space with metric $\rho$, $f: X \to X$ a homeomorphism on $X$ and $\mathbb{Z}^+$ the set of positive integers. Note that $f$ defines a $\mathbb{Z}$-action: $(n,x) \mapsto f^n(x)$ for every $n \in \mathbb{Z}$ and $x \in X$. However, the definitions and most of the results given in this section for $\mathbb{Z}$-actions hold good even if $f$ is a (non-invertible) continuous function, in which case the action is by the semigroup of non-negative integers.

For $n \in \mathbb{N}$, the \textit{metric}
\[
\rho_n(x, y) = \max_{0 \leq k \leq n-1} \rho(f^k(x), f^k(y))
\]
measures the maximum separation between the first $n$ iterates of $x$ and $y$ for any $n\in \mathbb{Z}^+$.

An $(n, \epsilon)$-covering of $X$, for $n \in \mathbb{Z}^+$ and $\epsilon > 0$,  is a collection of sets whose union is $X$, and the $\rho_n$-diameter of each of them is less than $\epsilon$. Let $cov(n, \epsilon, f)$ denote the minimum of the cardinalities of all $(n, \epsilon)$-coverings of $X$. Since $X$ is compact, $cov(n, \epsilon, f)$ is well defined.

The topological entropy of \(f\), denoted by \(h(f)\), is defined as:
\[h(f) = \lim_{\epsilon \to 0^+} \limsup_{n \to \infty} \frac{1}{n} \log(\mathrm{cov}(n, \epsilon, f)).\]
This topological invariant measures the ``complexity" of the orbit structure of $f$ in the sense that it measures the exponential growth rate of the number of essentially different orbit segments of length $n$. The entropy can also be expressed equivalently in terms of $sep(n,\epsilon,f)$ and $span(n,\epsilon, f)$ as described below.

Fix $n \in \mathbb{Z}^+$ and $\epsilon > 0$. A set $E \subset X$ is called $(n, \epsilon)$-separated if for any $x, y \in E$ with $x \neq y$, we have $\rho_n(x, y) \ge \epsilon$. In other words, any two distinct points in $E$ must have their orbits diverge by at least $\epsilon$ within the first $n$ iterations. Let $sep(n, \epsilon, f)$ denote the maximum of the cardinalities of all $(n, \epsilon)$-separated sets. A set $E \subset X$ is called $(n, \epsilon)$-spanning if for any $x \in X$, there exists $y \in E$ such that $\rho_n(x, y) < \epsilon$. This means that for any point in $X$, we can find a point in $E$ whose orbit stays within $\epsilon$ of the orbit of $x$ for the first $n$ iterations. Let $span(n, \epsilon, f)$ denote the minimum of the cardinalities of all $(n, \epsilon)$-spanning sets. Again since $X$ is compact, $sep(n,\epsilon,f)$ and $span(n,\epsilon, f)$ exist.

It can be shown that $ \mathrm{cov}(n, 2\epsilon, f) \leq \mathrm{span}(n, \epsilon, f) \leq \mathrm{sep}(n, \epsilon, f) \leq \mathrm{cov}(n, \epsilon, f)$ for any $n\in \mathbb{Z}^+$ and $\epsilon>0$. Hence 
\[\begin{aligned}
h(f) &= \lim_{\epsilon\to 0^+}\limsup_{n\to\infty}\frac{1}{n}\log\big(\mathrm{sep}(n,\epsilon,f)\big)\\
     &= \lim_{\epsilon\to 0^+}\limsup_{n\to\infty}\frac{1}{n}\log\big(\mathrm{span}(n,\epsilon,f)\big).
\end{aligned}
\]

The topological entropy of a homeomorphism $f: X \to X$ is independent of the metric chosen to generate the topology of $X$ and it is preserved under topological conjugacy. Several structural properties hold: for iterates, $h(f^m) = m \cdot h(f)$ when $m \in \mathbb{N}$, and if $f$ is invertible, then $h(f^{-1}) = h(f)$; so in general $h(f^m) = |m| \cdot h(f)$ for $m \in \mathbb{Z}$. Moreover, if $X$ is the union of finitely many closed forward $f$-invariant subsets, then $h(f)$ equals the maximum of the entropies restricted to these subsets. For product systems, entropy behaves additively, i.e., $h(f \times g) = h(f) + h(g)$, and for factor maps, entropy decreases, meaning if $g$ is a factor of $f$, then $h(f) \geq h(g)$. 

A $\mathbb{Z}^d$-action on $X$ is a continuous map $T: \mathbb{Z}^d \times X \to X$, i.e., $T^0(x) = x$ and $T^{m_1+m_2}(x) = T^{m_1}(T^{m_2}(x))$ for all $x \in X$ and $m_1, m_2 \in \mathbb{Z}^d$, where $T^m(x)$ denotes $T(m,x)$. 
We use the following definitions as given by Oprocha \cite{oprocha2007chain}, Shah and Das \cite{shah2015note, shah2015different}. For $k \in \{1, 2, 3, \dots, 2^d\}$, let $k^b$ represent $k-1$ in the $d$-positional binary system, i.e., 
$k-1 = \sum_{i=1}^d k_i^b 2^{i-1}$, where $k^b \in \{0, 1\}^d$. For $x, y \in \mathbb{Z}^d$, we say $x >^k y$ if $(-1)^{k_i^b} x_i > (-1)^{k_i^b} y_i \quad \forall i$, where $x = (x_1, \dots, x_d)$ and $y = (y_1, \dots, y_d)$. By $x \geq ^k y$, we mean $x >^k y$ or $x=y$. We also use the notation $x <^k y$ and $x \leq^k y$ to mean that $y >^k x$ and $y \geq^k x$ respectively. Finally, whenever we write $m \geq^k 0$, we mean that $m \in \mathbb{Z}^d$ and $0$ is the identity element of the group $\mathbb{Z}^d$; this abuse of notation by replacing $(0,0,...,0) \in \mathbb{Z}^d$ by $0$ doesn't lead to any confusion, as the context makes it clear.


\section{$k$-type Topological Entropy}
In this section, we will define $k$-type entropy of $\mathbb{Z}^d$ actions, analogous to the entropy of $\mathbb{Z}$ actions. We start with the definition of $k$-type metric followed by the definitions of  $cov(n,k,\epsilon,T)$, $span(n,k,\epsilon,T)$ and $sep(n,k,\epsilon,T)$, and finally the $k$-type entropy. Throughout this section, $n$ and $d$ denote positive integers, $k \in \{1,2,...,2^d\}$ and $\epsilon > 0$. Also, $T$ is a $\mathbb{Z}^d$-action on a compact metric space $X$.

\begin{definition}
The $k$-type metric $\rho_{n,k}$ is given by
\[ \rho_{n,k}(x, y) = \max_{\substack{||m|| < n \\  m\ge^k 0}} \rho(T^m(x), T^m(y)), \]
where $\|m\| = \max\limits_{1 \leq i \leq n} |m_i|$ and $|m_i|$ is the absolute value of $m_i$.
\end{definition}

\begin{definition}
    A collection of subsets of $X$ is called an $(n, k, \epsilon)$-covering of $X$ with respect to $T$ if the $\rho_{n,k}$ diameter of each set in the collection is less than $\epsilon$ and $X$ equals the union of all these sets. Let $cov(n,k,\epsilon,T)$ denote the minimum of cardinalities of all $(n,k,\epsilon)$-coverings.
\end{definition}

\begin{definition}
    An $(n, k, \epsilon)$-spanning set $E \subset X$ is a set such that for every $x \in X$, there exists $y \in E$ with $\rho_{n,k}(x, y) < \epsilon$. The minimum of cardinalities of all such sets is denoted by $span(n, k, \epsilon, T)$.
\end{definition}

\begin{definition}
     An $(n, k, \epsilon)$-separated set $E \subset X$ is a set such that for any distinct $x, y \in E$, we have $\rho_{n,k}(x, y) \ge \epsilon$. The maximum of cardinalities of all such sets is denoted by $sep(n, k, \epsilon, T)$.
\end{definition}
Note that the compactness of $X$ ensures that all these three numbers $cov(n,k,\epsilon,T)$, $span(n,k,\epsilon,T)$ and $sep(n,k,\epsilon,T)$ exist. Now, we can define the $k$-type topological entropy. 

\begin{definition}
    The \emph{$k$--type topological entropy} of $T$ (hereafter called $k$-type entropy)  is defined as
    \[
        h_k(T) = \lim_{\epsilon \to 0^+} \; \limsup_{n \to \infty} \frac{1}{n} \log \operatorname{cov}(n, k, \epsilon, T).
    \]
\end{definition}

\begin{remark}
    For each fixed $\epsilon > 0$, the quantity
    \[
        \limsup_{n \to \infty} \frac{1}{n} \log \operatorname{cov}(n, k, \epsilon, T)
    \]
    is non-increasing as $\epsilon \to 0^+$. Consequently, the limit in the definition exists, and $h_k(T)$ is well-defined for every $\mathbb{Z}^d$--action $T$.
\end{remark}

\begin{proposition}
For a $\mathbb{Z}^d$--action $T$ on a compact metric space $(X,\rho)$, for every $n\in\mathbb{N}$, $k \in \{1,\dots,2^d\}$, and $\epsilon>0$, we have
\[
\operatorname{cov}(n,k,2\epsilon,T)\;\le\;\operatorname{span}(n,k,\epsilon,T)\;\le\;\operatorname{sep}(n,k,\epsilon,T)\;\le\;\operatorname{cov}(n,k,\epsilon,T).
\]
\end{proposition}

\begin{proof}

Let $E$ be an $(n,k,\epsilon)$-spanning set with cardinality equal to $span(n,k,\epsilon, T)$. For every $x\in X$ there exists $y\in E$ with $\rho_{n,k}(x,y)<\epsilon$. Hence the $\epsilon$-balls $\{B_{\rho_{n,k}}(y,\epsilon):y\in E\}$ cover $X$ and thus is an $(n,k,2\epsilon)$-cover of $X$, since each ball has diameter at most $2\epsilon$. Hence $cov(n,k,2\epsilon,T) \le span(n,k,\epsilon,T)$.
 
Let $F$ be an $(n,k,\epsilon)$-separated set with cardinality equal to $sep(n,k,\epsilon, T)$. By maximality, for every $x\in X$ there exists $y\in F$ such that $\rho_{n,k}(x,y)<\epsilon$, otherwise $F\cup \{x\}$ will be a larger separated set. Hence $F$ is an $(n,k,\epsilon)$-spanning set. Therefore, $span(n,k,\epsilon,T) \le sep(n,k,\epsilon,T)$.

Let $\mathcal{U}$ be an $(n,k,\epsilon)$-cover of $X$ with cardinality equal to $cov(n,k,\epsilon,T)$. If $E$ is an $(n,k,\epsilon)$-separated set, then each member of  $\mathcal{U}$ can contain at most one point of $E$, since otherwise two distinct points of $E$ would lie in the same set of diameter less than $\epsilon$. Thus $sep(n,k,\epsilon,T) \le cov(n,k,\epsilon,T)$.
\end{proof}

\begin{remark}
    $h_k(T)$ can be equivalently defined using spanning sets or separated sets based on the $k$-type metric and the above inequality. \begin{align*}
h_k(T) &= \lim_{\epsilon \to 0^+} \limsup_{n \to \infty} 
          \frac{1}{n} \log \, \mathrm{span}(n, k, \epsilon, T) \\
       &= \lim_{\epsilon \to 0^+} \limsup_{n \to \infty} 
          \frac{1}{n} \log \, \mathrm{sep}(n, k, \epsilon, T).
\end{align*}
\end{remark}

\begin{remark}
If $d=1$, then $f(x)=T(1,x)$ for every $x\in X$ is a homeomorphism. It follows that $h_1(T)=h(f)$ because $k-1=0$ and thus $n_1<^kn_2$ if and only if $n_1<n_2$ for any $n_1,n_2 \in \mathbb{Z}$. Similarly, we have $h_2(T)=h(f^{-1})$. However $h(f)=h(f^{-1})$ implies that $h_1(T)=h_2(T)=h(f)$.

The core idea is that these definitions measure the exponential growth rate of distinguishable orbits under the $\mathbb{Z}^d$-action $T$, taking the partial order $\le^k$ into account.
\end{remark}

With the above definition of $k$-type topological entropy, we have the following results. On some occasions, we work with two different metrics on the same space and in such cases, we denote $cov(n,k,\epsilon,T)$, $span(n,k,\epsilon,T)$, $sep(n,k,\epsilon,T)$ and $h_k(T)$ by $cov(n,k,\epsilon,T,\rho)$, $span(n,k,\epsilon,T,\rho)$, $sep(n,k,\epsilon,T,\rho)$ and $h_k(T, \rho)$ respectively, where $\rho$ is the metric with respect to which these numbers are calculated.

\begin{theorem}\label{MetricIndependence}
Let $T : \mathbb{Z}^d \times X \to X$ be a $\mathbb{Z}^d$-action on a compact metric space $X$. 
If $\rho$ and $\rho'$ are equivalent metrics on $X$, then 
\[
h_k(T,\rho) \;=\; h_k(T,\rho') \qquad \text{for all } k \in \{1,2,\dots,2^d\}.
\]
\end{theorem}

\begin{proof}
For $\epsilon > 0$, define $\delta(\epsilon) \;=\; \sup\{\rho'(x,y) : \rho(x,y) < \epsilon\}.$
Since $\rho$ and $\rho'$ are equivalent on the compact space $X$, we have $\delta(\epsilon) \to 0$ as $\epsilon \to 0$.  
Suppose $E \subseteq X$ is an $(n,k,\epsilon)$-spanning set for $T$ with respect to $\rho$. Then for every $x \in X$ there exists $y \in E$ such that $\rho_{n,k}(x,y) < \epsilon$. By definition of $\delta(\epsilon)$, this implies $\rho'_{n,k}(x,y) < \delta(\epsilon)$. Hence, $E$ is also an $(n,k,\delta(\epsilon))$-spanning set for $T$ with respect to $\rho'$. Consequently,
$\operatorname{span}(n,k,\delta(\epsilon),T,\rho') \;\le\; \operatorname{span}(n,k,\epsilon,T,\rho).$
This gives us, 
$h_k(T,\rho') \;\le\; h_k(T,\rho).$
The reverse inequality follows similarly by interchanging $\rho$ and $\rho'$. Thus,
$h_k(T,\rho) \;=\; h_k(T,\rho')$
for all $k \in \{1,2,\dots,2^d\}$.
\end{proof}

\begin{Definition}
Let $T_{1} : \mathbb{Z}^{d} \times X \to X$ and $T_{2} : \mathbb{Z}^{d} \times Y \to Y$ be $\mathbb{Z}^{d}$--actions on the compact metric spaces $X$ and $Y$, respectively.  
We say that $(X,T_{1})$ and $(Y,T_{2})$ are \emph{topologically conjugate} if there exists a homeomorphism $\pi : X \to Y$ such that
$\pi \circ T_{1}^{n} \;=\; T_{2}^{n} \circ \pi,  \text{for all } n \in \mathbb{Z}^{d}$.
In this case, the map $\pi$ is called a \emph{conjugacy} between $(X,T_{1})$ and $(Y,T_{2})$.
If $\pi$ is only a continuous surjection, then $(Y,T_{2})$ is called a \emph{factor} of $(X,T_{1})$, and $\pi$ is a \emph{factor map}.
\end{Definition}

\begin{theorem}[Conjugacy Invariance]
Let $T_1 : \mathbb{Z}^d \times X \to X$ and $T_2 : \mathbb{Z}^d \times Y \to Y$ be $\mathbb{Z}^d$-actions on compact metric spaces $(X,\rho_X)$ and $(Y,\rho_Y)$, respectively. If $T_1$ and $T_2$ are topologically conjugate, then \[h_k(T_1) \;=\; h_k(T_2) \qquad \text{for all } k \in \{1,2,\dots,2^d\}.\]
\end{theorem}

\begin{proof}
Since $T_1$ and $T_2$ are topologically conjugate, there exists a homeomorphism $h : X \to Y$ such that $h \circ T_1^{m} \;=\; T_2^{m} \circ h$ for all  $m \in \mathbb{Z}^d.$
Define a new metric $\rho'_Y$ on $Y$ by $\rho'_Y(y_1,y_2) \;=\; \rho_X\!\big(h^{-1}(y_1), h^{-1}(y_2)\big)$.
Because $h$ is a homeomorphism, $\rho'_Y$ is equivalent to $\rho_Y$.
For the induced metrics, we compute
\begin{align*}
\rho'_{Y,n,k}(y_1,y_2) 
&= \max_{m\geq^k 0, \,\|m\| < n} \rho'_Y\big(T_2^{m}(y_1), T_2^{m}(y_2)\big) \\
&= \max_{m\geq^k 0, \,\|m\| < n} \rho_X\!\Big(h^{-1}(T_2^{m}(y_1)), \, h^{-1}(T_2^{m}(y_2))\Big) \\
&= \max_{m\geq^k 0, \,\|m\| < n} \rho_X\!\Big(T_1^{m}(h^{-1}(y_1)), \, T_1^{m}(h^{-1}(y_2))\Big) \\
&= \rho_{X,n,k}\!\big(h^{-1}(y_1), h^{-1}(y_2)\big).
\end{align*}
Thus, $h^{-1}$ is an isometry between $(Y,\rho'_{Y,n,k})$ and $(X,\rho_{X,n,k})$. Consequently, $\operatorname{span}(n,k,\epsilon,T_2,\rho'_Y) \;=\; \operatorname{span}(n,k,\epsilon,T_1,\rho_X)$ and thus, $h_k(T_2,\rho'_Y) \;=\; h_k(T_1,\rho_X)$.
By Theorem \ref{MetricIndependence}, $h_k(T_2,\rho_Y) = h_k(T_2,\rho'_Y)$ and hence, $h_k(T_2) \;=\; h_k(T_1)$.
\end{proof}

\begin{theorem}\label{EntropyFactorMap}
Let $T_{1} : \mathbb{Z}^{d} \times X \to X$ and $T_{2} : \mathbb{Z}^{d} \times Y \to Y$ be $\mathbb{Z}^{d}$--actions on compact metric spaces $X$ and $Y$, respectively and let $k\in\{1,\dots,2^d\}$. If $(Y,T_{2})$ is a factor of $(X,T_{1})$ via a continuous surjection $\pi: X \to Y$, then $h_k(T_2) \leq h_k(T_1)$.
\end{theorem}

\begin{proof}
Since $\pi$ is uniformly continuous, for every $\varepsilon>0$ there exists $\delta>0$ such that $\rho_X(x,x')<\delta$ implies $\rho_Y(\pi(x),\pi(x'))<\varepsilon$. Let $S_Y \subset Y$ be an $(n,k,\varepsilon)$--separated set for $T_2$. For each $y \in S_Y$ choose $x_y \in \pi^{-1}(y)$ and set $S_X := \{x_y : y \in S_Y\} \subset X$.

Suppose $y \neq y'$ in $S_Y$ and $\rho_{n,k}(x_y, x_{y'}) < \delta$. Then $\rho(T_1^m(x_y), T_1^m(x_{y'})) < \delta$ for all $m$ with $\|m\|<n$ and $m \ge^k 0$, which by uniform continuity of $\pi$ gives $\rho(T_2^m(y), T_2^m(y')) = \rho(\pi(T_1^m(x_y)), \pi(T_1^m(x_{y'}))) < \varepsilon$ for all $m$. Hence $\rho_{n,k}(y,y') < \varepsilon$, contradicting that $S_Y$ is $(n,k,\varepsilon)$--separated. Thus $S_X$ is $(n,k,\delta)$--separated for $T_1$, and therefore $\mathrm{sep}(n,k,\varepsilon,T_2) \le \mathrm{sep}(n,k,\delta,T_1)$. Taking limits as in the definition of entropy yields $h_k(T_2) \le h_k(T_1)$.
\end{proof}

\begin{remark}
    A $\mathbb{Z}^d$--action $T:\mathbb{Z}^d \times X \to X$ is said to be an \emph{isometry} if, for every $m \in \mathbb{Z}^d$, the map $T(m,-):X \to X$ is an isometry. 
    In this case, the $k$--type topological entropy is zero for every $k$.
\end{remark}

If $Y$ is a subspace of $X$ such that $T^m(y)\in Y$ for every $y\in Y$ and every $m\in\mathbb{Z}^d$, then $Y$ is called a $T$-invariant subspace and further, if $Y$ is also closed, then $(Y,T)$ is called a subsystem of $(X,T)$, in which case $(Y,T)$ itself can be considered as a dynamical system in its own respect. The following theorem shows that the $k$-type entropy of $(X,T)$ is equal to the maximum of the $k$-type entropies of subsystems whose union is equal to $X$.

\begin{theorem}\label{EntropyUnionsInvariantSubsets}
Let $T : \mathbb{Z}^d \times X \to X$ be a $\mathbb{Z}^d$-action on a compact metric space $X$.  
Suppose $A_1, \dots, A_\ell$ are closed (not necessarily disjoint) $T$-invariant subsets such that
\[
X = \bigcup_{i=1}^\ell A_i.
\]
Then
\[
h_k(T) = \max_{1 \leq i \leq \ell} h_k\!\left(T|_{A_i}\right)
\qquad \text{for all } k \in \{1,2,\dots,2^d\}.
\]

\end{theorem}

\begin{proof}
We first prove the inequality 
\[
h_k(T) \;\geq\; \max_{1 \leq i \leq \ell} h_k(T|_{A_i}).
\]
Indeed, if $E \subset A_i$ is an $(n,k,\varepsilon)$-separated set for $T|_{A_i}$, then $E$ is also $(n,k,\varepsilon)$-separated for $T$ on $X$. Hence
$\operatorname{sep}(n,k,\varepsilon,T|_{A_i}) \;\leq\; \operatorname{sep}(n,k,\varepsilon,T).$
Passing to logarithms, dividing by $n$, and taking $\limsup_{n \to \infty}$ followed by $\varepsilon \to 0$, we obtain
$h_k(T|_{A_i}) \;\leq\; h_k(T).$
Taking the maximum over $i$ gives the desired inequality.

Conversely, we show that
\[
h_k(T) \;\leq\; \max_{1 \leq i \leq \ell} h_k(T|_{A_i}).
\]
Let $\operatorname{span}_i(n,k,\varepsilon,T)$ denote the minimum of cardinalities of all $(n,k,\varepsilon)$-spanning set for $T|_{A_i}$.  
If $F_i$ is such a spanning set for $A_i$, then
$F = \bigcup_{i=1}^\ell F_i$
is an $(n,k,\varepsilon)$-spanning set for $X$. Thus
\[
\operatorname{span}(n,k,\varepsilon,T) \;\leq\; \sum_{i=1}^\ell \operatorname{span}_i(n,k,\varepsilon,T)
\;\leq\; \ell \cdot \max_{1 \leq i \leq \ell} \operatorname{span}_i(n,k,\varepsilon,T).
\]
It follows that
\begin{align*}
h_k(T) 
&= \lim_{\varepsilon \to 0} \limsup_{n \to \infty} \frac{1}{n} 
   \log \operatorname{span}(n,k,\varepsilon,T) \\
&\leq \lim_{\varepsilon \to 0} \limsup_{n \to \infty} \frac{1}{n}
   \log \!\bigg( \ell \cdot \max_{1 \leq i \leq \ell} \operatorname{span}_i(n,k,\varepsilon,T) \bigg) \\
&= \max_{1 \leq i \leq \ell} \lim_{\varepsilon \to 0} \limsup_{n \to \infty}
   \frac{1}{n} \log \operatorname{span}_i(n,k,\varepsilon,T) \\
&= \max_{1 \leq i \leq \ell} h_k(T|_{A_i}).
\end{align*}
Combining the two inequalities yields the claimed equality.\end{proof}

\begin{corollary}
    Let $T : \mathbb{Z}^d \times X \to X$ be a $\mathbb{Z}^d$-action on a compact metric space $X$. If $A \subseteq X$ is a closed $T$-invariant subset, then $h_k(T|_A) \leq h_k(T).$
\end{corollary}
\begin{proof}
    If $A \subseteq X$ is closed and $T$-invariant, apply the above theorem with $A_1 = A$ and $A_2 = X$, which gives $h_k(T) = \max\big(h_k(T|_A),\, h_k(T)\big),$ hence $h_k(T|_A) \leq h_k(T)$.
\end{proof}

The following theorem gives the $k$-type entropy of ``product" of two $\mathbb{Z}^d$-actions in terms of $k$-type entropy of individual actions. Given two $\mathbb{Z}^d$-actions $T_1$ and $T_2$ on two compact metric spaces $X$ and $Y$ respectively, we define the $\mathbb{Z}^d$-action $T_1\times T_2$ on $X\times Y$ as $T_1\times T_2 (m,(x,y)) = (T_1^m(x),T_2^m(y))$ for every $m\in \mathbb{Z}^d$ and $(x,y)\in X\times Y$. We calculate the $k$-type entropy of $T_1\times T_2$ with respect to the metric $\rho$ on $X\times Y$ defined as $\rho((x_1,y_1),(x_2,y_2))= \max \{ \rho_X(x_1,x_2),\rho_Y(y_1,y_2)\}$, where $\rho_X$ and $\rho_Y$ are the metrics on $X$ and $Y$ respectively. We use the same notations as given in this paragraph for the next theorem and also its proof.

\begin{theorem}\label{ProductFormula}
For every $k\in\{1,\dots,2^d\}$, the $k$--type topological entropy is additive:
\[h_k(T_1\times T_2) = h_k(T_1) + h_k(T_2).\]

\end{theorem}

\begin{proof}

Note that $k$-type metric on $X\times Y$ is \[
\rho_{n,k}\big((x_1,y_1),(x_2,y_2)\big)
=\max\{\rho_{X,n,k}(x_1,x_2),\,\rho_{Y,n,k}(y_1,y_2)\}.
\]

\medskip\noindent
Let $E_1\subset X$ and $E_2\subset Y$ be $(n,k,\varepsilon)$--spanning sets for $T_1$ and $T_2$, respectively. For any $(x,y)\in X\times Y$ choose $x'\in E_1$, $y'\in E_2$ with
$\rho_{X,n,k}(x,x')<\varepsilon$ and $\rho_{Y,n,k}(y,y')<\varepsilon$. Then
\[
\rho_{n,k}((x,y),(x',y'))=\max\{\rho_{X,n,k}(x,x'),\rho_{Y,n,k}(y,y')\}<\varepsilon,
\]
so $E_1\times E_2$ is an $(n,k,\varepsilon)$--spanning set for $T_1\times T_2$. Hence
\[
\operatorname{span}(n,k,\varepsilon,T_1\times T_2)\le 
\operatorname{span}(n,k,\varepsilon,T_1)\cdot\operatorname{span}(n,k,\varepsilon,T_2).
\]
Passing to logarithms, dividing by $n$, taking $\limsup_{n\to\infty}$ and letting $\varepsilon\to 0^+$ yields
\[
h_k(T_1\times T_2)\le h_k(T_1)+h_k(T_2).
\]

\medskip\noindent  
Conversely, let $F_1\subset X$ and $F_2\subset Y$ be $(n,k,\varepsilon)$--separated sets for $T_1$ and $T_2$, respectively. For distinct $(x,y),(x',y')\in F_1\times F_2$ either $x\neq x'$ or $y\neq y'$, and consequently
\[
\rho_{n,k}((x,y),(x',y'))=\max\{\rho_{X,n,k}(x,x'),\rho_{Y,n,k}(y,y')\}\ge\varepsilon.
\]
Thus $F_1\times F_2$ is $(n,k,\varepsilon)$--separated for $T_1\times T_2$, and
\[
\operatorname{sep}(n,k,\varepsilon,T_1\times T_2)\ge
\operatorname{sep}(n,k,\varepsilon,T_1)\cdot\operatorname{sep}(n,k,\varepsilon,T_2).
\]
Taking logarithms, dividing by $n$, applying $\limsup_{n\to\infty}$ and letting $\varepsilon\to 0^+$ yields
\[
h_k(T_1\times T_2)\ge h_k(T_1)+h_k(T_2).
\]

\noindent Combining the two inequalities gives the desired equality.
\end{proof}

Let $T$ be a $\mathbb{Z}^d$-action on $X$, and let $r \in \mathbb{Z}^d$.  We define the \emph{$r^{th}$ iterate} $T^r$ by  $T^r(m,x) \;=\; T^{\,m \star r}(x),$ for every $m \in \mathbb{Z}^d$ and $x\in X,$ 
where $m \star r$ denotes the coordinate-wise product, i.e., if  
$m = (m_1,\dots,m_d)$ and $r = (r_1,\dots,r_d)$, then  
$m \star r = (m_1 r_1, \dots, m_d r_d).$ 

For $i\in\{1,2,\dots,d\}$, let $e_i\in\mathbb{Z}^d$ denote the $i$-th standard basis vector (with $1$ in the $i$-th coordinate and $0$ elsewhere). 
$T^{e_i}$ can be considered as a $\mathbb{Z}$-action also by defining $T^{e_i}(l,x)=T^{le_i}(x)$ for every $l\in\mathbb{Z}$ and $x\in X$. While calculating the $k$-type entropy of $T^{e_i}$ as a $\mathbb{Z}^d$-action, we consider the iterates $T^{m\star e_i}$, where $m\geq^k0$ and $||m||<n$ to find the $k$-type metric $\rho_{n,k}$. However, the set of iterates $\{ T^{m\star e_i} | m\geq^k0, ||m||<n  \}$ is same as either $\{ T^{le_i} | o\leq l < n \}$ or $\{ T^{le_i} | -n< l \leq 0\}$ depending on  $k$. The latter two sets of iterates determine the $\rho_n$ metric for $T^{e_i} $ and $T^{-e_i} $ as $\mathbb{Z}$-actions respectively.
Hence it follows that the $k$-type entropy of $T^{e_i} $ as a $\mathbb{Z}^d$-action for any $k\in\{1,2,\dots,2^d\}$ is same as $h(T^{e_i})$ or $h(T^{-e_i})$. However, $h(T^{e_i})=h(T^{-e_i})$. Hence $h_k(T^{e_i}) = h(T^{e_i})$ and thus $h_k(T^{le_i}) = h(T^{le_i})$ for every $l\in\mathbb{Z}$, where $T^{e_i}$ and $T^{le_i}$ are considered as $\mathbb{Z}^d$-actions on the left hand side and as $\mathbb{Z}$-actions on the right hand side. Further, it is well known that $h(T^{le_i}) = |l|\cdot h(T^{e_i})$ and hence we have $h_k(T^{le_i}) = |l|\cdot h_k(T^{e_i})$ for the $\mathbb{Z}^d$-action $T^{e_i}$.


\begin{theorem}\label{EntropyIterates}
For any $k\in\{1,2,\dots,2^d\}$,
\[
h_k(T^r) \;\geq\; \max\limits_{1\leq i\leq n}\{|r_i| \cdot h_k(T^{e_i})\}.
\]
\end{theorem}

\begin{proof}
Let $A$ be a minimal spanning set for $T^r$.  
Then, for every $x \in X$, there exists $y \in A$ such that  
\[\rho_{n,k}(T^r(x),\, T^r(y)) = \max\limits_{m\geq^k 0, \,\|m\| < n}\{\rho(T^{m\star r}(x),\, T^{m\star r}(y)) \} < \epsilon.\]
Now, \begin{align*}
\rho_{n,k}\big(T^{r_i e_i}(x),\,T^{r_i e_i}(y)\big)
  &= \max_{\substack{m \geq^k 0 \\ \|m\| < n}}
     \left\{\rho\big(T^{m \star r_i e_i}(x),\, T^{m \star r_i e_i}(y)\big)\right\} \\[6pt]
  &= \max_{\substack{m \geq^k 0 \\ \|m\| < n}}
     \left\{\rho\big(T^{m_i r_i e_i}(x),\, T^{m_i r_i e_i}(y)\big)\right\}.
\end{align*}
Note that $m_i r_i e_i = m_i e_i\star r$, implying that \[\{m_i r_i e_i |  m \geq^k 0,  \|m\| < n\} \subset \{m\star r |  m \geq^k 0,  \|m\| < n\}. \]
Hence, for every $x \in X$, there exists $y \in A$ such that  $\rho_{n,k}(T^{r_i e_i}(x),\, T^{r_i e_i}(y)) < \epsilon$ for all $i \in \{1,2,\dots,d\}$.
Hence $A$ is also a spanning set for $T^{r_i e_i}$.  
Since a minimal spanning set for $T^{r_i e_i}$ can only be smaller, we have $\operatorname{span}(T^r) \geq \operatorname{span}(T^{r_i e_i})$ for all $i$.
Thus, $\operatorname{span}(T^r) \geq \max\limits_{1\leq i\leq n} \{\operatorname{span}(T^{r_i e_i})\}$.
Taking limits, we obtain $h_k(T^r) \geq \max\limits_{1\leq i\leq n}\{h_k(T^{r_i e_i})\} = \max\limits_{1\leq i\leq n}\{|r_i| \cdot h_k(T^{e_i})\}$.
\end{proof}

\section{$k$-type entropy of toral automorphisms}

A hyperbolic matrix refers to an integer matrix of determinant one, with dstinct real eigenvalues, each of them having absolute value not equal to one. In this section, we will calculate the $k$-type entropy of $\mathbb{Z}^2$-actions defined by two commuting hyperbolic matrices on a two-dimensional torus. We define a two-dimensional torus as $\mathbb{T}^2 = \mathbb{R}^2 / \mathbb{Z}^2$ and if $A$ is a hyperbolic matrix, then the map $f_A(x) = Ax$ is an automorphism, called a hyperbolic toral automorphism. It is well known that $h(f_A)=|\lambda|$, where $\lambda$ is the eigenvalue of $A$ with $|\lambda|>1$ (see \cite{brin2002introduction} for a proof). We call this eigenvalue $\lambda$ as the expanding value of $A$ and the other eigenvalue, which is $\frac{1}{\lambda}$, as its contracting eigenvalue.
 
Now, let $A$ and $B$ be two commuting hyperbolic matrices.
Then they admit common eigenvectors, say $v_1$ and $v_2$, where we assume that the expanding eigenvalues of $A$ and $B$, say $\lambda_A$ and $\lambda_B$ correspond to $v_1$, and the contracting eigenvalues $\lambda_A^{-1}$ and $ \lambda_B^{-1}$ correspond to $v_2$.

These define a natural $\mathbb{Z}^2$-action on the torus:
\[
T : \mathbb{Z}^2 \times \mathbb{T}^2 \to \mathbb{T}^2, 
\qquad T((m_1,m_2),x) = A^{m_1}B^{m_2}x.
\]
Our aim is to compute the $k$-type entropy of this action.
The proofs of our statements in this section i.e., Proposition \ref{prop:k14}, Proposition \ref{prop:k23} and Theorem \ref{thm:torus} are similar to the proof of Proposition 2.6.1 in \cite{brin2002introduction} that computes the topological entropy of a hyperbolic toral automorphism.

Let $\pi : \mathbb{R}^2 \to \mathbb{T}^2$ be the quotient map. 
For $x,y \in \mathbb{R}^2$, write $x-y = a_1 v_1 + a_2 v_2$ and define $\tilde{\rho}(x,y) = \max(|a_1|,|a_2|).$
This is a translation-invariant metric on $\mathbb{R}^2$, and it induces a metric $\rho$ on $\mathbb{T}^2$ via $\pi$. 
A $\tilde{\rho}$-ball of radius $\epsilon$ is a parallelogram with sides of length $2\epsilon$ parallel to $v_1,v_2$. 
The $\tilde\rho_{n,k}$-balls are also parallelograms, but with different side lengths as described in the following propositions.

\begin{proposition}\label{prop:k14}
For $k \in \{1,4\}$, a $\tilde\rho_{n,k}$-ball of radius $\epsilon$ is a parallelogram with side lengths 
\[
2\epsilon |\lambda_A|^{-(n-1)}|\lambda_B|^{-(n-1)} 
\quad\text{and}\quad 2\epsilon.
\]
\end{proposition}
\begin{proof}
Since we have, $x-y = a_1 v_1 + a_2 v_2,$ for some $a_1,a_2\in\mathbb R$, under the linear map $A^{m_1}B^{m_2}$ the coordinates scale as
\[
A^{m_1}B^{m_2}(x)-A^{m_1}B^{m_2}(y)
= (a_1 \lambda_A^{m_1}\lambda_B^{m_2})\,v_1 + (a_2 \lambda_A^{-m_1}\lambda_B^{-m_2})\,v_2.
\]
Hence, in the metric $\tilde\rho$,
\[
\tilde\rho\big(A^{m_1}B^{m_2}x,\,A^{m_1}B^{m_2}y\big)
= \max\big\{\,|a_1|\;|\lambda_A|^{m_1}|\lambda_B|^{m_2},\;
|a_2|\;|\lambda_A|^{-m_1}|\lambda_B|^{-m_2}\big\}.
\]
For $k=1$ the $k$-type metric is the maximum over $0\le m_1,m_2\le n-1$, so
\[
\tilde\rho_{n,k}(x,y)
= \max_{0\le m_1,m_2\le n-1}\max\big\{\,|a_1|\;|\lambda_A|^{m_1}|\lambda_B|^{m_2},\;
|a_2|\;|\lambda_A|^{-m_1}|\lambda_B|^{-m_2}\big\}.
\]
The first term is increasing in each $m_i$ and attains its maximum at $m_1=m_2=n-1$, while the second term is decreasing in each $m_i$ and attains its maximum at $m_1=m_2=0$. Therefore
\[
\tilde\rho_{n,k}(x,y)=\max\big\{\,|a_1|\;|\lambda_A|^{\,n-1}|\lambda_B|^{\,n-1},\;|a_2|\,\big\}.
\]
The condition $\tilde\rho_{n,k}(x,y)<\epsilon$ is thus equivalent to the two inequalities
\[
|a_1| < \epsilon\,|\lambda_A|^{-(n-1)}|\lambda_B|^{-(n-1)},\qquad
|a_2| < \epsilon.
\]
Thus, a $\tilde\rho_{n,k}$-ball of radius $\epsilon$ is a parallelogram with side lengths 
\[
2\epsilon |\lambda_A|^{-(n-1)}|\lambda_B|^{-(n-1)} 
\quad\text{and}\quad 2\epsilon.
\]
This proves the claim. Similar proof works for $k=4$, where the maximum is taken over the indices $-(n-1)\leq m_1,m_2\leq 0$.
\end{proof}

\begin{proposition}\label{prop:k23}
For $k \in \{2,3\}$, a $\tilde\rho_{n,k}$-ball of radius $\epsilon$ is a parallelogram 
with side lengths 
\[
2\epsilon |\lambda_A|^{-(n-1)} 
\quad\text{and}\quad 2\epsilon|\lambda_B|^{-(n-1)}.
\]
\end{proposition}
\begin{proof}
As before write $x-y=a_1v_1+a_2v_2$. For the sign choices corresponding to $k\in\{2,3\}$ the $k$-type metric ranges over exponents where one coordinate uses positive powers and the other uses negative powers. Concretely, for $k=2$ one obtains the family $A^{m_1}B^{-m_2}$, $0\le m_1,m_2\le n-1$, and the scaled coordinates become $(a_1 \lambda_A^{m_1}\lambda_B^{-m_2})\,v_1 + (a_2 \lambda_A^{-m_1}\lambda_B^{m_2})\,v_2.$
Thus
\[
\tilde\rho_{n,k}(x,y)
= \max_{0\le m_1,m_2\le n-1}\max\Big\{\,|a_1|\; \frac{|\lambda_A|^{m_1}}{|\lambda_B|^{m_2}},\;
|a_2|\; \frac{|\lambda_B|^{m_2}}{|\lambda_A|^{m_1}}\,\Big\}.
\]
For fixed $m_1$ the first factor is decreasing in $m_2$, while for fixed $m_2$ it is increasing in $m_1$; hence the maximum over the rectangle $0\le m_1,m_2\le n-1$ is attained at the corner $m_1=n-1$, $m_2=0$, giving the value $|a_1|\,|\lambda_A|^{\,n-1}$. Similarly, the second factor attains its maximum at $m_1=0$, $m_2=n-1$, giving $|a_2|\,|\lambda_B|^{\,n-1}$. Therefore
\[
\tilde\rho_{n,k}(x,y)=\max\{\,|a_1|\,|\lambda_A|^{\,n-1},\;|a_2|\,|\lambda_B|^{\,n-1}\,\}.
\]
The inequality $\tilde\rho_{n,k}(x,y)<\epsilon$ is equivalent to
\[
|a_1| < \epsilon\,|\lambda_A|^{-(n-1)},\qquad
|a_2| < \epsilon\,|\lambda_B|^{-(n-1)}.
\]
Hence a $\tilde\rho_{n,k}$-ball of radius $\epsilon$ is a parallelogram 
with side lengths 
\[
2\epsilon |\lambda_A|^{-(n-1)} 
\quad\text{and}\quad 2\epsilon|\lambda_B|^{-(n-1)}.
\]
The proof for $k=3$ is similar, where the maximum is taken over the indices $-(n-1)\leq m_1,m_2\leq 0$.
\end{proof}

\begin{theorem}\label{thm:torus}
For any $k$, the $k$-type entropy of $T$ is
\[
h_k(T) = \log|\lambda_A| + \log|\lambda_B|.
\]
\end{theorem}
\begin{proof}
From Proposition \ref{prop:k14} and Proposition \ref{prop:k23}, it follows that a $\tilde\rho_{n,k}$-ball of radius $\epsilon$ is a parallelogram of maximum area 
$4 \epsilon^2 |\lambda_A|^{-(n-1)}|\lambda_B|^{-(n-1)}$. 
Since the induced metric $\rho$ on $\mathbb{T}^2$ is locally isometric to $\tilde\rho$, for sufficiently small $\epsilon$, the area of a $\rho_{n,k}$-ball of radius $\epsilon$ in $\mathbb T^2$ is also at most $4 \epsilon^2 |\lambda_A|^{-(n-1)}|\lambda_B|^{-(n-1)}$.
Considering the torus with unit area, packing such disjoint balls gives $\text{cov}(n,\epsilon,T) \geq \frac{|\lambda_A|^{n-1}|\lambda_B|^{n-1}}{\epsilon^2}.$ Taking limits gives $h_k(T) \ge \log|\lambda_A| + \log|\lambda_B|$. 

For the upper bound, note that $\rho_{n,k}$-balls are parallelograms that tile the plane up to constants depending only on the eigenbasis. 
Hence the torus can be covered by at most $C|\lambda_A|^{n-1}|\lambda_B|^{n-1}/\epsilon^2$ such balls, where $C$ is a constant depending on the angle between the eigenvectors $v_1$ and $v_2$. 
Taking limits yields $h_k(T) \le \log|\lambda_A| + \log|\lambda_B|$. 
\end{proof}




\textbf{Acknowledgements.} The first author gratefully acknowledges the financial support from the Council of Scientific and Industrial Research (CSIR), Government of India, through the fellowship file no. 09/1026(0044)/2021-EMR-I.

\end{document}